\documentclass[journal]{IEEEtran}
\usepackage[english]{babel}
\usepackage[usenames,dvipsnames,table]{xcolor}
\usepackage{amsmath}
\usepackage{amsthm}
\usepackage{amssymb}
\usepackage{graphicx}
\usepackage{caption}
\usepackage{subcaption}
\usepackage{url}
\usepackage{float}
\usepackage{tikz}
\usepackage{siunitx}
\usetikzlibrary{shapes,arrows,positioning}
\usepackage{pgfplots} 		       % plot tikz figures
\usepackage[utf8]{inputenc}
\usepackage[hidelinks]{hyperref}
\usepackage{dsfont}   		       % to get a nice \1 
\usepackage[mathscr]{euscript}
\usepackage{mathtools}
\usepackage{xargs} 		       % more than one optionalargument	
\usepackage[ruled,shortend,commentsnumbered]{algorithm2e}
\usepackage{pdfpages} 				% allows to include pdf's

\graphicspath{{./img/}{./}} % add a graphicspath

\newcommand{\w}{\omega}

\newcommand{\1}{\mathds 1}

\newcommand{\p}{\partial}

\newcommand{\vp}{\varphi}

\newcommand{\n}{\nabla}
\newcommand{\G}{\mathcal{G} }
\newcommand{\V}{\mathcal{V} }

\newcommand{\E}{\mathcal{E} }

\newcommand{\mysum}[1][]{#1 \sum}%\textstyle

\newcommand{\real}{\mathbb{R}}
\newcommand{\realnonnegative}{{\mathbb{R}}_{\ge 0}}
\newcommand{\norm}[1]{\ensuremath{\| #1 \|}}
\newcommand{\until}[1]{[#1]}
\newcommand{\longthmtitle}[1]{\mbox{}\textup{\textsl{(#1):}}}

\newcommand{\map}[3]{#1:#2 \rightarrow #3}
\newcommand{\st}{\operatorname{subject \text{$\, \,$} to}}
\newcommand{\minimize}{\operatorname{minimize}}
\newcommand{\oprocendsymbol}{\hbox{$\bullet$}}
\newcommand{\oprocend}{\relax\ifmmode\else\unskip\hfill\fi\oprocendsymbol}

\renewcommand{\l}{\lambda}
\newcommand{\optimizer}[1][b]{P_{g}^{\text{opt}}(#1)}
\newcommand{\pproj}[2]{[#1]^+_{#2}}

\DeclareMathOperator{\col}{col}

\DeclareMathOperator{\diag}{diag}

\DeclareMathOperator{\blockdiag}{blockdiag}
\DeclareMathOperator*{\argmax}{arg\,max}

\DeclareMathOperator{\Sin}{\mathbf{sin}}
\DeclareMathOperator{\Cos}{\mathbf{cos}}

\newtheorem{theorem}{Theorem}[section]
\newtheorem{proposition}[theorem]{Proposition}

\newtheorem{corollary}[theorem]{Corollary}
\theoremstyle{remark}
\newtheorem{remark}[theorem]{Remark}
\theoremstyle{definition}

\newtheorem{definition}[theorem]{Definition}

\makeatletter
\newcommand{\removelatexerror}{\let\@latex@error\@gobble}
\makeatother

\usepackage[normalem]{ulem} % strike out text, example:  \sout{Hello}
\definecolor{new}{rgb}{0.55,0,0.55}
\makeatletter
\newcommand{\thickhline}{%
  \noalign {\ifnum 0=`}\fi \hrule height 1pt
  \futurelet \reserved@a \@xhline
}
\newcolumntype{"}{@{\hskip\tabcolsep\vrule width 1pt\hskip\tabcolsep}}
\makeatother

\usepackage[prependcaption,colorinlistoftodos]{todonotes}

\allowdisplaybreaks

\begin{document}

% \title{Coupling of price-bidding electricity markets with physical
% power networks\thanks{A preliminary version of this work has been
% submitted to the 2018 Power Systems Computation Conference
% as~\cite{PSCC18-TS-AC-CDP-AS-JC}.}}
% \title{Real-time market-based dispatch in power networks: competition and (frequency) stability\thanks{A preliminary version of this work has been submitted to the 2018 Power Systems Computation Conference as~\cite{PSCC18-TS-AC-CDP-AS-JC}.}}

\title{Frequency-driven market mechanisms for optimal dispatch in
  power networks\thanks{A preliminary version of this work has been
    submitted to the 2018 Power Systems Computation Conference
    as~\cite{PSCC18-TS-AC-CDP-AS-JC}.}}

\author{Tjerk~Stegink,~Ashish~Cherukuri,~Claudio~De~Persis,~Arjan~van~der~Schaft,~and~Jorge
  Cort\'es% <-this % stops a space
  % <-this % stops a space
  \thanks{This work is supported by the NWO (Netherlands Organisation
    for Scientific Research) \emph{Uncertainty Reduction in Smart
      Energy Systems} (URSES) programme and the ARPA-e \emph{Network
    Optimized Distributed Energy Systems} (NODES) program.}
  \thanks{T. W. Stegink and C. De Persis and A. J. van der Schaft are
    affiliated with the Jan C. Willems Center for Systems and Control,
    University of Groningen, the Netherlands.  {\tt\small
      \{t.w.stegink,
      c.de.persis,a.j.van.der.schaft\}@rug.nl}}% {\tt\small t.w.stegink@rug.nl}}% Nijenborgh 4
  % \thanks{A.J. van der Schaft is with the Johann Bernoulli Institute
  % for Mathematics and Computer Science, University of Groningen,
  % Nijenborgh 9, 9747 AG Groningen,
  % the Netherlands.  {\tt\small a.j.van.der.schaft@rug.nl}}%
  \thanks{A. Cherukuri is with the Automatic Control Laboratory, ETH
    Z\"{u}rich.  {\tt\small cashish@control.ee.ethz.ch}}
  \thanks{J. Cort\'es is with the Department of Mechanical and
    Aerospace Engineering, University of California, San Diego.
    {\tt\small cortes@ucsd.edu}} }% <-this % stops a space

% The paper headers
% \markboth{Journal of \LaTeX\ Class Files,~Vol.~14, No.~8, August~2015}%
% {Shell \MakeLowercase{\textit{et al.}}: Bare Demo of IEEEtran.cls for IEEE Journals}
% The only time the second header will appear is for the odd numbered pages
% after the title page when using the twoside option.
% 
% *** Note that you probably will NOT want to include the author's ***
% *** name in the headers of peer review papers.                   ***
% You can use \ifCLASSOPTIONpeerreview for conditional compilation here if
% you desire.

% If you want to put a publisher's ID mark on the page you can do it like
% this:
% \IEEEpubid{0000--0000/00\$00.00~\copyright~2015 IEEE}
% Remember, if you use this you must call \IEEEpubidadjcol in the second
% column for its text to clear the IEEEpubid mark.

% use for special paper notices
% \IEEEspecialpapernotice{(Invited Paper)}

% make the title area
\maketitle

% ABSTRACT
\begin{abstract}
  This paper studies real-time bidding mechanisms for economic
  dispatch and frequency regulation in electrical power networks. We
  consider a market administered by an independent system operator
  (ISO) where a group of strategic generators participate in a
  Bertrand game of competition.  Generators bid prices at which they
  are willing to produce electricity. Each generator aims to maximize
  their profit, while the ISO seeks to minimize the total generation
  cost and to regulate the frequency of the system.  We consider a
  continuous-time bidding process coupled with the swing dynamics of
  the network through the use of frequency as a feedback signal for
  the negotiation process. We analyze the stability of the resulting
  interconnected system, establishing frequency regulation and the
  convergence to a Nash equilibrium and optimal generation levels. The
  results are verified in the IEEE 14-bus benchmark case.
\end{abstract}

% Note that keywords are not normally used for peerreview papers.
% \begin{IEEEkeywords}
%   IEEE, IEEEtran, journal, \LaTeX, paper, template.
% \end{IEEEkeywords}

% For peer review papers, you can put extra information on the cover
% page as needed:
% \ifCLASSOPTIONpeerreview
% \begin{center} \bfseries EDICS Category: 3-BBND \end{center}
% \fi
% 
% For peerreview papers, this IEEEtran command inserts a page break and
% creates the second title. It will be ignored for other modes.
\IEEEpeerreviewmaketitle

%%%%%%%%%%%%%%%%%%%%%%%% INTRODUCTION %%%%%%%%%%%%%%%%%%%%%%%%%%
\section{Introduction}
% \marginjc{I think this intro paragraph puts the emphasis on the wrong
%   things. We do not deal much with aggregation. Instead, we deal with
%   breaking the hierarchy in a novel way: integrating market mechanisms
%   with physical dynamics of the power network.}

%   With the goal of integrating distributed energy sources (DERs)
%   into the grid, a hierarchical structure for the electricity market
%   is envisioned. At the lower layer, groups of DERs coordinate their
%   response under an generator (which acts as a large-scale producer
%   of electricity), while at the top level the independent system
%   operator (ISO) interacts with the generators to achieve cost
%   efficiency and stable operation of the grid
%   \cite{gkatzikis2013role}. In this set-up, generators compete with
%   each other to maximize their individual profit, while the ISO can
%   exploit the flexibility of the generators for adjusting the
%   generation levels to compensate for imbalances in the network. In
%   this paper we study the competition aspect among the generators in
%   combination with the challenge of achieving economic efficiency
%   and frequency stabilization.

Power generation dispatch is typically done in a hierarchical fashion,
where the different layers are separated according to their time
scales. Broadly, at the top layer economic efficiency is ensured via
market clearing and at the bottom layer frequency control and
regulation is achieved via primary and secondary controllers. However,
the intermittent and uncertain nature of distributed energy resources
(DERs) and their integration into the power grid represents a major
challenge to the current design. Of particular concern is the need to
maintain both frequency regulation and cost efficiency of regulation
reserves in the face of increasing fluctuations in renewables. To this
end, we propose an integrated dynamic market mechanism which combines
the real-time market and frequency regulation, allowing competitive
market players, including renewable generation, to negotiate
electricity prices while using the most recent information on the grid
frequency.

\subsubsection*{Literature review}
The combination of economic dispatch and frequency regulation has
received increasing attention in recent years. Various works have
sought to move beyond the traditional and compartmentalized
hierarchical control layers to instead simultaneously achieve
frequency stabilization and economic dispatch in power
networks~\cite{trip2016internal,zhangpapaautomatica,li2016connecting}
and microgrids \cite{cady2015distributed,dorfler2016breaking}. Along
this line of research, the various agents involved work cooperatively
towards the satisfaction of a common goal.  An alternative body of
research has investigated the use of price-based incentives for
economic generation- and demand-side management and frequency
regulation~\cite{alv_meng_power_coupl_market,DJS-MC-AMA:16,stegink2017unifying}.
To achieve these goals, these works consider dynamic pricing
mechanisms in conjunction with system dynamics of the power
network. We also adopt this approach, with the important distinction
that here we allow generators to bid in the market (hence, they are
price-setters rather than price-takers). This viewpoint results in a
Bertrand game of competition among the generators.  Our previous
work~\cite{cherukuri2017iterative,cherukuri2016decentralized} studied
this type of games established that iterative bidding can achieve
convergence to an optimal allocation of power generation, without
considering the effects on the dynamics of the power network. The
underlying assumption was that generation setpoints could be commanded
after convergence, which in practice poses a limitation, considering
the fast time-scales at which DERs operate.  Instead, this paper
proposes an online bidding scheme where the setpoints are updated
continuously throughout time to better cope with fast changes in the
network. In this way, we tackle simultaneously both frequency
regulation, optimal power dispatch and the competitive aspect among
the generators.

\subsubsection*{Statement of contributions}
We consider an electrical power network consisting of an independent
system operator (ISO) and a group of competitive generators. Each
generator seeks to maximize its individual profit, while the ISO aims
to solve the economic dispatch problem and regulate the
frequency. Since the generators are not willing to share their cost
functions, the ISO is unable to solve the economic dispatch
problem. Instead, it has the generators compete in a bidding market
where they submit bids to the ISO in the form of a price at which they
are willing to produce electricity. In return, the ISO determines the
power generations levels the generators have to meet. We analyze the
underlying Bertrand game among the generators and characterize the
Nash equilibria that correspond to optimal power dispatch termed
\emph{efficient Nash equilibria}. In particular, we establish the
existence of such efficient Nash equilibria and provide a sufficient
condition for its uniqueness.  We also propose a Nash equilibrium
seeking scheme in the form of a continuous-time bidding process that
captures the interaction between the generators and the ISO. In this
scheme, the generators adjust their bid based on their current bid and
the production level that the ISO requests from them with the aim to
maximize their profit.  At the same time, the ISO adjusts the
generation setpoints to minimize the total payment to the generators
while taking the power balance and frequency deviation into
account. Moreover, along the execution of the algorithm the
nonnegativity constraints on the bids and power generation quantities
are satisfied.  The use of the local frequency error as a feedback
signal in the negotiation process couples the ISO-generator
coordination scheme with the swing dynamics of the power network. We
show that each equilibrium of the interconnected system corresponds to
an efficient Nash equilibrium, optimal generation levels and zero
frequency regulation. We furthermore establish local convergence to
such an equilibrium by invoking a suitable invariance principle for
the closed-loop projected dynamical system.  Finally, the numerical
results on the IEEE 14-bus benchmark show fast convergence of the
closed-loop system to an optimal equilibrium, even under sudden
changes of the load and the cost functions.

%%%%%%%%%%%%%%%%%%%%%%%%%%%%%%%%%%%%%%%%%%%%%%%%%%
% \subsubsection*{Organization}
% The remainder of the paper is organized as
% follows. Section~\ref{sec:phys-power-netw} introduces the notation
% and the dynamic model of the power network.
% Section~\ref{sec:problem-statement} presents the problem statement
% and Section \ref{sec:exist-uniq-nash} discusses the existence and
% uniqueness of efficient Nash equilibria. Section
% \ref{sec:continuous-real-time} introduces the ISO-generator bidding
% scheme and shows the local convergence to the efficient equilibria
% of the interconnection with the swing equations.  Simulations
% illustrate the results in Section ... Finally, we gather our
% conclusions and ideas for future work in Section
% \ref{sec:conclusions}.
%%%%%%%%%%%%%%%%%%%%%%%%%%%%%%%%%%%%%%%%%%%%%%%%%%

\subsubsection*{Notation} 
Let $\mathbb R , \mathbb R_{\ge0} , \mathbb R_{>0}$
be the set of real, nonnegative real, and positive real numbers, respectively.
%For $m \in \integerspositive$, we denote $\until{m} = \{1, \dots, m\}=\mathcal I_m$.
We write the set $\{1,\ldots,n\}$ compactly as $\until{n}$. We denote
by $\1\in\mathbb R^n$ the vector whose elements are equal to 1. Given
a twice differentiable function $\map{f}{\real^n}{\real}$, its
gradient and its Hessian evaluated at $x$ is written as $\nabla f(x)$
and $\n^2f(x)$,
respectively. %In addition, we write $\nabla_if(x)=\frac{\p f}{\p x_i}(x)$. % The Hessian of a
% twice-differentiable function $f:\mathbb R^n\to\mathbb R$ is denoted
% by $\n^2f$.  For $v\in\mathbb R^n$, we let $\ddiag{v}\in\mathbb
% R^{n\times n}$ denote the diagonal matrix with entries
% $v_1,\ldots,v_n$ on the diagonal.
A twice continuously differentiable function $\map{f}{\real^n}{\real}$
is \emph{strongly convex} on $S\subset \real^n$ if it is convex and,
for some $\mu>0$, its Hessian satisfies $\n^2f(x)>\mu I$ for all $x\in
S$. For scalars $a,b\in\mathbb R$ we denote by $\pproj{a}{b}$ the
operator
\begin{align}\label{eq:projection}
  \pproj{a}{b}=
  \begin{cases}
    a & \text{if } b>0\\
    \max(a,0) & \text{if } b=0.
  \end{cases}
\end{align}
For vectors $a,b\in\real^n$, $\pproj{a}{b}$ denotes the vector whose $i$-th element is given by $\pproj{a_i}{b_i}$ for $i\in[n]$. For $A\in\mathbb R^{m\times n}$, the induced $2$-norm is denoted by $\norm{A}$. Given $v\in\mathbb R^n,\tau\in\mathbb R^{n\times n}$, we write $\|v\|_\tau:=\sqrt{v^T\tau v}$.  Given a set of numbers $v_1,v_2,\ldots,v_n \in \real$, $\col(v_1,\ldots,v_n)$ denotes the column vector $\begin{bmatrix} v_1, \dots, v_n \end{bmatrix}^T $ and likewise $\diag(v_1,\ldots,v_n)$ denotes the $n\times n$ diagonal matrix with entries $v_1,\ldots,v_n$ on the diagonal. For $u,v\in\mathbb R^n$ we write $u\perp v$ if $u^Tv=0$. We use the compact notational form $0\leq u\perp v\geq 0$ to denote the complementarity conditions $u\geq0,v\geq 0, u\perp v$. The notations $\Sin(.)$ and $\Cos(.)$ are used to represent the element-wise sine and cosine functions respectively.

% For a vector $v\in\mathbb R^n$ we write
% $[v]=\diag_{i\in\{1,\ldots,n\}}\{v_i\}$.

\section{Power network model and dynamics}\label{sec:phys-power-netw}
% Here we present basic concepts on the dynamics of an electrical power
% network. 
We consider an electrical power network consisting of $ n $ buses and
$m$ transmission lines. The network is represented by a connected and
undirected graph $ \G = (\V, \E) $, where nodes $ \V =\until{n}$
represent buses and edges $ \E \subset \V \times \V $ are the
transmission lines connecting the buses.  The edges are arbitrarily
labeled with a unique identifier in $[m]$ and the ends of each edge
are arbitrary labeled with ‘+’ and ‘-’.  The incidence matrix $D \in
\real^{n \times m}$ of the resulting directed graph~is
\begin{align*}
  D_{ik}=
  \begin{cases}
    +1 &\text{if $i$ is the positive end of edge $k$},\\
    -1 &\text{if $i$ is the negative end of edge $k$},\\
    0 & \text{otherwise.}
  \end{cases}
\end{align*}
Each bus $i$ represents a control area and is assumed to have one
generator and a load~$P_{di}$.  The dynamics at the buses is assumed
to be governed by the \emph{swing equations} \cite{powsysdynwiley},
given by
\begin{equation}
  \begin{aligned}
    \dot \delta&=\w\\
    M\dot \w&= -D\Gamma\Sin (D^T\delta)-A\w+P_g-P_d
  \end{aligned}\label{eq:swingdeltacomp}
\end{equation}
with $P_d=\col(P_{d1},\ldots,P_{dn})$.  Here
$\Gamma=\diag(\gamma_1,\ldots,\gamma_m)$, where
$\gamma_k=B_{ij}V_{i}V_{j}=B_{ji}V_{i}V_{j}$ and $k\in[m]$ corresponds
to the edge between nodes $i$ and~$j$.  Table~\ref{tab:par3SG}
presents a list of symbols employed in the
model~\eqref{eq:swingdeltacomp}.
\begin{table}[htb]\vspace*{-3mm}
  \begin{align*}
    \delta&\in \mathbb R^n && \text{voltage phase angle}  \\
    \w&\in\mathbb R^n & &  \text{frequency deviation w.r.t. the nominal frequency}  \\
    P_{g}&\in\realnonnegative^n & & \text{power generation} \\
    P_{d}&\in\realnonnegative^n  && \text{power load}  \\
    M&\in\mathbb R_{\ge0}^{n\times n} & & \text{diagonal matrix of
      moments of inertia}
    \\
    A&\in\realnonnegative^{n\times n}& & \text{diagonal matrix of
      asynchronous damping constants}
    \\
    V_{i}&\in\mathbb R_{>0}& & \text{voltage magnitude at bus } i
    \\ %\hline
    B_{ij}&\in\mathbb R_{>0}& & \text{negative of the susceptance of
      transmission line } (i,j)
  \end{align*}%\vspace*{-6mm}
  \caption{State variables and parameters of swing
    equations~\eqref{eq:swingdeltacomp}.}\label{tab:par3SG}
\end{table}

For the stability analysis carried out later, it is convenient to work
with the voltage phase angle differences $\vp=D_t^T\delta\in\mathbb
R^{n-1}$.  Here $D_t\in\mathbb R^{n\times (n-1)}$ is the incidence
matrix of an arbitrary tree graph on the set of buses $\until{n}$
(e.g., a spanning tree of the physical network). Furthermore, let $
U(\vp)=-\1^T\Gamma\Cos(D^TD_t^{\dagger T}\vp)$, where
$D_t^\dagger=(D_t^TD_t)^{-1}D_t^T$ denotes the Moore-Penrose
pseudo-inverse of $D_t$. Then the physical
system~\eqref{eq:swingdeltacomp} in the $(\vp,\w)$-coordinates takes
the form
\begin{equation}\label{eq:swingeqU}
  \begin{aligned}
    \dot \vp&=D_t^T\w\\
    M\dot \w&=-D_t\nabla U(\vp)-A\w+P_g-P_d,
  \end{aligned}
\end{equation}
where we observe that $D_tD_t^{\dagger}D=(I-\frac1n\1\1^T)D=D$.

\section{Problem description}\label{sec:problem-statement}
In this section we formulate the problem statement, introduce the
necessary game-theoretic tools and discuss the goals of the paper.  

\subsection{ISO-generator coordination}
Taking as starting point the electrical power network model described
in Section~\ref{sec:phys-power-netw}, here we outline the elements of
the ISO-generator coordination problem following the exposition
of~\cite{cherukuri2017iterative,cherukuri2016decentralized}.  Let
$\map{C_i}{\realnonnegative}{\realnonnegative}$ be the cost incurred
by generator $i \in \until{n}$ in producing $P_{gi}$ units of
power. We assume $C_i$ is strongly convex on the domain
$\realnonnegative$ and satisfies $\nabla C_i(0)\geq0$. Given the total
network cost
\begin{align}\label{eq:total-cost}
  C(P_g) :=\mysum_{i\in\until{n}}C_i(P_{gi})
\end{align}
and a power load $P_d$, the ISO seeks to solve the \emph{economic
dispatch (ED)} problem
\begin{subequations}\label{eq:ISO-OPF}
  \begin{align}
    \minimize & \quad C(P_g),
    \\
    \st & \quad \1^TP_g=\1^TP_d,\label{eq:ISO-pow-bal}
    \\
    & \quad P_g \geq 0,\label{eq:Pg-ge0}
  \end{align}
\end{subequations}
and, at the same time, to regulate the frequency of the physical power
network.  We assume the total load to be positive, i.e., $\1^TP_d>0$
such that \eqref{eq:ISO-OPF} is feasible. Since the constraints
\eqref{eq:ISO-pow-bal}~\eqref{eq:Pg-ge0} are affine, Slater's
condition holds implying that \eqref{eq:ISO-OPF} has zero duality
gap. We can also show that its primal-dual optimizer
$(P_g^*,\l^*,\mu^*)$ is unique by exploiting strong convexity of
$C$. We assume that for the power injection $P_g=P_g^*$, there exists
an equilibrium $(\bar \vp,\bar \w)$ of~\eqref{eq:swingeqU} that
satisfies $D^TD_t^{\dagger T}\bar\vp\in(-\pi/2,\pi/2)^m$. The latter
assumption is standard and is referred to as the \emph{security
  constraint} in the power systems literature \cite{powsysdynwiley}.

We note that the ISO cannot determine the optimizer of the ED problem
\eqref{eq:ISO-OPF} because generators are strategic and they do not
reveal their cost functions to anyone. Instead, the ISO operates a
market where each generator $i\in\until{n}$ submit a bid $b_i \in
\realnonnegative$ in the form of a price at which it is willing to
provide power.  Based on these bids, the ISO aims to find the power
allocation that meets the load and minimizes the total payment to the
generators.  Thus instead of solving the ED problem \eqref{eq:ISO-OPF}
directly, the ISO considers, given a bid $b \in \realnonnegative^n$,
the convex optimization problem
%
% \begin{subequations}\label{eq:ISOprob}
%  \begin{align}
%    \underset{P_g\ge0}{\minimize} & \quad b^TP_g  \label{eq:ISOobj-fun} 
%    \\
%    \st & \quad \1^TP_g=\1^TP_d . \label{eq:pow-bal-con}
%  \end{align}
%\end{subequations}
%
\begin{subequations}\label{eq:ISOprob}
  \begin{align}
    \minimize & \quad b^TP_g, \label{eq:ISOobj-fun} 
    \\
    \st & \quad \1^TP_g=\1^TP_d, \label{eq:pow-bal-con}
    \\
    & \quad P_g \ge 0.
  \end{align}
\end{subequations}
A fundamental difference between~\eqref{eq:ISO-OPF}
and~\eqref{eq:ISOprob} is that the latter optimization is linear and
may in general have multiple solutions.  Let $\optimizer$ be the
optimizer of \eqref{eq:ISOprob} the ISO selects given bids~$b$ and
note that this might not be unique. Knowing the ISO's strategy,
%
% \margints{Mention here Stackelberg game:?  Generators are leaders,
% ISO is follower. }
% 
each generator $i$ bids a quantity $b_i\geq0$ to maximize its payoff
\begin{equation}\label{eq:payoffgen}
  \Pi_i(b_i,P_{gi}^{\text{opt}}(b)):=
  P_{gi}^{\text{opt}}(b)b_i-C_i(P_{gi}^{\text{opt}}(b)),
\end{equation}
where $P_{gi}^{\text{opt}}(b)$ is the $i$-th component of the
optimizer $P_g^{\text{opt}}(b)$.  Note that this function is not
continuous in the bid~$b$.  Since each generator is strategic, we
analyze the market clearing, and hence the dispatch process explained
above using tools from game theory~\cite{TB-GJO:82,DF-JT:91}.

\subsection{Inelastic electricity market game}
We define the \emph{inelastic electricity market game} as
\begin{itemize}
\item Players: the set of generators $\until{n}$.
\item Action: for each player $i$, the bid $b_i\in\realnonnegative$.
\item Payoff: for each player $i$, the payoff $\Pi_i$ defined in
  \eqref{eq:payoffgen}.
\end{itemize}
\noindent In the sequel we interchangeably use the notation
$b\in\realnonnegative^n$ and $(b_i,b_{-i})\in\realnonnegative^n$ for
the bid vector, where $b_{-i}\in\realnonnegative^{n-1}$ represents the
bids of all players except $i$. We note that the payoff of generator
$i$ not only depends on the bids of the other players but also on the
optimizer $P_g^{\text{opt}}(b)$ the ISO selects. Therefore, the
concept of a Nash equilibrium is defined slightly differently compared
to the usual one.
\begin{definition}[Nash equilibrium \cite{cherukuri2016decentralized}]
  A bid profile $b^*\in\realnonnegative^n$ is a \emph{Nash equilibrium}
  of the inelastic electricity market game if there exists an optimizer
  $P_{g}^{\text{opt}}(b^*)$ of~\eqref{eq:ISOprob} such that for each
  $i\in\until{n}$,
\begin{align*}
  \Pi_i(b_i,P_{gi}^{\text{opt}}(b_i,b_{-i}^*))\leq
  \Pi_i(b_i^*,P_{gi}^{\text{opt}}(b^*))
\end{align*}
for all $b_i\in\realnonnegative$ with $ b_i\neq b_i^*$ and all
optimizers $P_{gi}^{\text{opt}}(b_i,b_{-i}^*)$ of \eqref{eq:ISOprob}
given bids $(b_i,b_{-i}^*)$.
\end{definition}
We are particularly interested in bid profiles for which the optimizer
of~\eqref{eq:ISO-OPF} is also a solution to~\eqref{eq:ISOprob}. This is
captured in the following definition.
\begin{definition}[Efficient bid and efficient Nash equilibrium]
  An \emph{efficient bid} of the inelastic electricity market is a bid
  $b^*\in\realnonnegative^n$ for which the optimizer $P_g^*$ of
  \eqref{eq:ISO-OPF} is also an optimizer of \eqref{eq:ISOprob}
  given bids $b=b^*$ and
  \begin{align}\label{eq:eff-bid}
    P_{gi}^*=\argmax_{P_{gi}\ge0}\{P_{gi} b^*_i-C_i(P_{gi})\}\ \text{ for each } i
    \in \until{n}.
  \end{align}
  A bid $b^*\in\realnonnegative^n$ is an \emph{efficient Nash
    equilibrium} of the inelastic electricity market game if it is an
  efficient bid and a Nash equilibrium.
\end{definition}

At the efficient Nash equilibrium, the optimizer of the ED problem
coincides with the production levels that maximize the individual
profits \eqref{eq:payoffgen} of the generators. This justifies
studying the efficient Nash equilibria.

\subsection{Paper objectives}
Given the problem setup described above, neither the ISO nor the
individual strategic generators are able to determine the efficient
Nash equilibrium a priori.  As a first objective, we are interested in
designing a Nash equilibrium seeking mechanism in the form of a
bidding process where the generators coordinate with the ISO to
dynamically update their bids and production levels, while respecting
the nonnegativity constraints throughout its execution.  Our second
objective is the characterization of the stability properties of the
interconnection of the bidding process
% between the bidding process that takes place to carry out the
% ISO-generator coordination and
with the physical dynamics of the power network.

% we start by showing the existence and uniqueness of the efficient
% Nash equilibrium of the inelastic electricity market game. We then
% propose a Nash equilibrium seeking bid update scheme and analyze its
% interconnection with the network dynamics.

%%%%%%%%%%%%%%%%%%%%%%%%%%%%%%%%%%%%%%%%%%%%%%%%%% 
\section{Existence and uniqueness of Nash
  equilibria}\label{sec:exist-uniq-nash}
%
% Here we establish the existence of an efficient Nash equilibrium of
% the inelastic electricity market game described in the previous
% section.  We also provide a condition under which it is unique.

In this section we establish existence of an efficient Nash
equilibrium and also provide a condition for its uniqueness.
% %
% \marginac{The following proposition is very similar to~\cite[Proposition
% 3.1]{cherukuri2016decentralized}. Therefore, we should highlight here what is
% different in the current result. Also, we can reduce the proof quite a bit
% by invoking certain assertions from that paper. The whole purpose is to
% avoid unnecessary repetition of arguments. 
% }
While~\cite{cherukuri2016decentralized} has established the existence
of one specific efficient Nash equilibrium, we provide in the
following result a characterization of all efficient Nash
equilibria.

\begin{proposition}\longthmtitle{Characterization of efficient Nash
    equilibria}\label{prop:exis-NE}
  Let $( P_{g}^*,\l^*,\mu^*)$ be the unique primal-dual optimizer of
  \eqref{eq:ISO-OPF},
  %
  % \marginac{We need to assume or state that the
  % problem~\eqref{eq:ISO-OPF} has zero duality gap. I'd prefer to
  % have it written in the previous section, in the paragraph
  % after~\eqref{eq:ISO-OPF}.}
  %	  
  that is, $P_{g}^*\in\real^n,\l^*\in\real,\mu^*\in\real^n$ satisfy
  the Karush-Kuhn-Tucker (KKT) conditions
  \begin{align}\label{eq:KKTcon-sc}
    \begin{aligned}
      \n C(P_g^*)&=\1\l^*+\mu^*, \qquad \1^TP_g^*=\1^TP_d, 
      \\
      0&\leq P_g^*\perp \mu^*\geq 0.
    \end{aligned}
  \end{align}
  %
  % \marginac{I would prefer to define the notation $\perp$ in the
  % prelims section.}
  %
  Suppose $P_{gi}^*>0$ for at least two distinct generators. Then, any
  $b^* \in \realnonnegative^n$ satisfying $\1\l^*\leq b^*\leq \nabla
  C(P_g^*)$ is an efficient Nash equilibrium of the inelastic
  electricity market game.
\end{proposition}
%  \blue{
%   \begin{remark}[Zero duality gap]
%     Since constraints \eqref{eq:ISO-pow-bal}, \eqref{eq:Pg-ge0}, are
%     affine and the feasibility set is nonempty, the refined Slater's
%     condition is satisfied for \eqref{eq:ISO-OPF} and hence, the
%     duality gap between the primal and the dual optimization
%     problems is zero \cite{SB-LV:04}. Under this condition, a $(
%     P_{g}^*,\l^*,\mu^*)$ is a primal-dual optimizer of
%     \eqref{eq:ISO-OPF} if and only if it satisfies the KKT
%     conditions \eqref{eq:KKTcon-sc}.
%   \end{remark}}
\begin{proof}
  Let $( P_{g}^*,\l^*,\mu^*)$ satisfy \eqref{eq:KKTcon-sc}, then in
  particular $\1\l^*\leq \nabla C(P_g^*)$.  Fix any bid $b^* \in
  \realnonnegative^n$ satisfying $\1\l^*\leq b^*\leq \nabla C(P_g^*)$. 
  % where we note there exists at least one such $b^*$.
  We will now prove that $b^*$ is efficient. Define $\nu^*:=b^*-\1\l^*$
  and note that $(P_g^*,\l^*,\nu^*)$ satisfies
  \begin{align}\label{eq:KKTcon-lin-opt2}
    \begin{aligned}
      b^*&=\1\l^*+\nu^*, \qquad \1^T P_g^*=\1^TP_d
      \\
      0&\leq P_g^* \perp \nu^*\geq 0.
    \end{aligned}
  \end{align}
  We note that Slater's condition holds for~\eqref{eq:ISOprob} and its
  KKT conditions are given by~\eqref{eq:KKTcon-lin-opt2}. Consequently,
  $P_g^*$ is a primal optimizer of~\eqref{eq:ISOprob}. In addition,
  the bid $b^*$ satisfies
  \begin{align}\label{eq:eff-bid2}
    P_{gi}^*&=\argmax_{P_{gi}\geq0}\{P_{gi}b_i^*-C_i(P_{gi})\}\ \text{
      for each } i\in\until{n}.
  \end{align}
  This is true as for each $i \in \until{n}$, the following optimality
  conditions
  \begin{align*}
    \nabla C_i(P_{gi}^*)&=b_i^*+\eta^*_i, \qquad 
    0\leq P_{gi}^*\perp  \eta^*_i \geq 0,
  \end{align*}
  are satisfied for $\eta^*_i=\nabla C_i(P_{gi}^*)-b^*_i$. Note that
  in the above set of conditions, $P_{gi}^* \eta^*_i = 0$ because if
  $P_{gi}^* >0$, then $\nabla C_i(P^*_{gi}) = \l^* = b^*_i$.  Thus, we
  have established that $b^*$ is efficient.  In the remainder of the
  proof we show that $b^*$ is a Nash equilibrium.  Suppose generator
  $i$ deviates from the bid $b_i^*$. We distinguish between two cases.
  Suppose first that $b_i>b_i^*$, then by replacing $b^*$ by
  $(b_i,b_{-i}^*)$ in~\eqref{eq:ISOprob} and checking the optimality
  conditions, we obtain $P_{gi}^{\text{opt}}(b_i,b_{-i}^*)=0$ as, by
  assumption, there is at least one other generator $j$ such that
  $b_j^*=\l^*<b_i$. Without loss of generality assume that $P^*_{gi} >
  0$ since otherwise $\Pi_i(b_i^*,P_{gi}^*) =
  \Pi_i(b_i,P_{gi}^{\text{opt}}(b_i,b_{-i}^*))$. For $P^*_{gi} > 0$,
  we have $b^*_i = \nabla C_i(P^*_{gi})$ and therefore $\n
  C_i(P_{gi})\leq b^*_i$ for all $P_{gi} \in [0,P^*_{gi}]$. As a
  result
  \begin{align*}
    \Pi_i(b_i,P_{gi}^{\text{opt}}(b_i,b_{-i}^*))=C(0)\leq \Pi_i(b_i^*,P_{gi}^*)
  \end{align*}
  % Using this fact, we get
  % \begin{align*}
  %   \Pi_i(b_i^*,P_{gi}^*(b^*))&=b_i^*P^*_{gi}-C_i(P^*_{gi})
  %   \\
  %   &=\int_0^{P^*_{gi}}(b_i^*-\nabla C_i(P_{gi}))dP_{gi}-C(0)
  %   \\
  %   & \geq -C(0)=\Pi_i(b_i,P_{gi}^{\text{opt}}(b_i,b_{-i}^*)).
  % \end{align*}
  This shows that a bid $b_i > b_i^*$ does not increase its payoff. 
  Suppose now that $b_i<b_i^*$, then
  \begin{align*}
    &\Pi_i(b_i,P_{gi}^{\text{opt}}(b_i,b_{-i}^*)) =
    b_iP_{gi}^{\text{opt}}(b_i,b_{-i}^*)) -
    C_i(P_{gi}^{\text{opt}}(b_i,b_{-i}^*))
    \\
    &\leq b_i^* P_{gi}^{\text{opt}}(b_i,b_{-i}^*)) -
    C_i(P_{gi}^{\text{opt}}(b_i,b_{-i}^*))
    \\
    &\leq b_i^*P^*_{gi}-C_i(P^*_{gi})=\Pi_i(b_i^*,P^*_{gi})
  \end{align*}
  where the second inequality follows from \eqref{eq:eff-bid2} as $b^*$
  is efficient. Hence, each generator $i$ has no incentive to deviate from
  bid $b_i^*$ given $b_{-i}^*$. We conclude that $b^*$ is an efficient
  Nash equilibrium of the inelastic electricity market game. 
\end{proof}

The proof of Proposition~\ref{prop:exis-NE} shows that if
$P_{gi}^*>0$, then generator $i$'s efficient Nash equilibrium bid
$b_i^*$ is equal to the (unique) Lagrange multiplier $\l^*$ associated
to the power balance \eqref{eq:ISO-pow-bal}.
%
% \marginac{Next, you might
% 	want to say the efficient Nash equilibrium instead of just Nash
% 	equilibirum no? Third, the efficient Nash equilibrium is defined
% 	by the bids and so saying “at the efficient Nash equilibirum, if
% 	$\bar P_{gi} > 0$” is kind of ambiguous. }
%
%  The underlying assumption here is that at least two generators have
%  a positive production at the optimal generation levels. We assume
%  this condition holds for the remainder of the paper unless stated
%  otherwise.
In the other case that $P_{gi}^*=0$, generator $i$'s Nash equilibrium
bid is larger than or equal to $\l^*$. This represents the case that
generator $i$'s marginal costs at zero power production is larger than
or equal to the market clearing price, and hence generator $i$ is not
willing to produce any electricity in that case. The underlying
assumption in Proposition~\ref{prop:exis-NE} is that at least two
generators have a positive production at the optimal generation
levels. We assume this condition holds for the remainder of the paper
unless stated otherwise.

% As observed from the proof of
% Proposition~\ref{prop:exis-NE}, if the primal-dual optimizer
% $(P_g^*,\l^*,\mu^*)$ satisfies $P_g^*>0$, then there is a unique
% efficient bid $b^*=\nabla C(P_g^*)=\1\l^*$.
The previous observations lead to the identification of the same
sufficient condition as in~\cite{cherukuri2016decentralized} to
guarantee the uniqueness of the efficient Nash equilibrium, which we
state here for completeness.

\begin{corollary}[Uniqueness of the efficient Nash
  equilibrium~\cite{cherukuri2016decentralized}]
  Let $( P_{g}^*,\l^*,\mu^*)$ be the primal-dual optimizer
  of~\eqref{eq:ISO-OPF}
  %
  % \marginac{Same comment about zero duality gap as written in the
  % previous Proposition}
  % 
  and suppose that $P_g^*>0$, then $b^*=\nabla C(P_g^*)=\1\l^*$ is the
  unique efficient Nash equilibrium of the inelastic electricity
  market game.
\end{corollary}

\begin{remark}[Any efficient Nash equilibrium is
  positive]\label{rem:pos-eff-NE}
  We observe from the optimality conditions \eqref{eq:KKTcon-sc} that,
  since $\1^TP_d>0$, and $P_g^*\geq0$, we must have that
  $P_{gi}^*>0$ and $\mu_i^*=0$ for some $i\in\until{n}$. As $\nabla
  C_i(P_{gi}^*)>0$ by the strict convexity of $C_i$ and the assumption
  $\nabla C_i(0)\ge0$, this implies that $\l^*>0$ and therefore also
  $b^*>0$. \oprocend
\end{remark}

\section{Interconnection of bid update scheme with power network
  dynamics}\label{sec:continuous-real-time}

In this section we introduce a Nash equilibrium seeking mechanism
between the generators and the ISO. Each generator dynamically updates
its bid based on the power generation setpoint received from the ISO,
while the ISO changes the power generation setpoints depending on the
generator bids and the frequency of the network. This update mechanism
of the bids and the setpoints is written as a continuous-time dynamical
system. We assume that each generator can only communicate with the ISO
and is not aware of the number of other generators participating, their
respective cost functions, or the load at its own bus. We study the
interconnection of the online bidding process with the power system
dynamics and establish local convergence to an efficient Nash
equilibrium, optimal power dispatch, and zero frequency deviation. 

\subsection{Price-bidding mechanism}\label{sec:bidding-scheme}

In our design, each generator $i\in\until{n}$ changes its bid $b_i \ge 0$
according to the projected dynamical system
\begin{subequations}\label{eq:bidding-process}
  \begin{align}\label{eq:produpdateb}
    \tau_{bi} \dot b_i&=\pproj{P_{gi}-\nabla C_i^*(b_i)}{b_i} ,
\end{align}
with gain $\tau_{bi}>0$. The projection operator in the above dynamics
ensures that trajectories starting in the nonnegative orthant remain
there.  
% Provided $b_i(0)\geq0$, the projection operator in the above
% dynamics ensures that $b_i(t) \ge 0$ for all $t \ge 0$ along the
% trajectory of~\eqref{eq:produpdateb}.  $\pproj{\cdot}{b_i}$ given by
% \eqref{eq:projection} ensures that $b(t)$ remains in the nonnegative
% orthant for all future time.
The map $\map{C_i^*}{\realnonnegative}{\realnonnegative}$ denotes the
\emph{convex conjugate} of the cost function $C_i$ and is defined as
\begin{align*}
  C_i^*(b_i):=\max_{P_{gi}\geq0}\{b_iP_{gi}-C_i(P_{gi})\}.
\end{align*}
Using tools from convex analysis~\cite[Section
I.6]{hiriart2013convex}, one can deduce that $C_i^*$ is convex and
continuously differentiable on the domain $\realnonnegative$ and
strictly convex on the domain $[\nabla C_i(0),\infty)$. Moreover, its
gradient satisfies $\nabla C_i^*(b_i)=\argmax_{P_{gi}\geq0}\{b_i
P_{gi}-C_i(P_{gi})\}$ for all $b_i\geq0$.

The motivation behind the update law \eqref{eq:produpdateb} is as
follows. Given the bid $b_i>0$, generator $i$ seeks to produce power
that maximizes its profit, which is given by
\begin{align*}
  P_{gi}^{\text{des}}=\nabla
  C_i^*(b_i)=\argmax_{P_{gi}\ge0}\{b_iP_{gi}-C_i(P_{gi})\}.
\end{align*}
However, if the ISO requests more power from the generator compared to
its desired quantity, i.e., $P_{gi}>P_{gi}^{\text{des}}$, then $i$
will increase its bid to increase its profit. On the other hand if
$P_{gi} < P_{gi}^{\text{des}}$, then $i$ will decrease its bid.
%It is interesting to notice that \eqref{eq:produpdateb} can be seen as
%a continuous-time version of the discrete-time dynamics proposed in
%\cite{cherukuri2017iterative,cherukuri2016decentralized}.

% Next, we provide an update law for the ISO depending on the bid
% $b\in\realnonnegative^n$ and the network frequency which is based on a
% projected primal-dual dynamics of \eqref{eq:ISOprob}, the ISO updates
% its actions according to the projected dynamical system

For the ISO we also provide an update law which depends on the
generator bids and the network frequency. This involves seeking a
primal-dual optimizer of \eqref{eq:ISOprob} or, equivalently, finding
a saddle-point of the augmented Lagrangian
\begin{align*}
  \mathcal L(P_g,\l)=b^TP_g+\l\1^T(P_d-P_g)+\rho\norm{\1^T(P_d-P_g)}^2 ,
\end{align*}
with parameter $\rho>0$. By writing the associated projected
saddle-point dynamics (see
e.g., \cite{cherukuri2017saddle,goebel2017stability}), the ISO
dynamics takes the form
% %
% \marginac{Possibly mention here or later after giving the ISO dynamics
% 	that one of the aims of the ISO is to drive the frequency
% deviation to zero. This shows up in the dynamics.}
%
% Based on the projected saddle-point dynamics of \eqref{eq:ISOprob} \cite{goebel2017stability},
% The ISO dynamics takes the form
%
% \marginac{We should explain in a sentence or two what is a primal-dual
% dynamics. Also give a reference.}
%
% By augmenting the Lagrangian of \eqref{eq:ISOprob},  the ISO dynamics takes the form 
%
\begin{equation}
  \begin{aligned}
    \tau_g\dot P_g&=\pproj{\1\l-b+\rho \1\1^T(P_d-P_g) -\sigma^2
      \w}{P_g} ,
    \\
    \tau_\l\dot \l&=\1^T(P_d-P_g) ,
  \end{aligned}\label{eq:ISOprimaldual}
\end{equation}
\end{subequations}
with design parameters $\sigma,\tau_\l\in\mathbb R_{>0}$ and diagonal
positive definite matrix $\tau_g\in\mathbb R^{n\times n}$. Bearing in
mind the ISO's second objective of driving the frequency deviation to
zero, we add the feedback signal $-\sigma^2\w$ to adjust the
generation based on the frequency deviation in the grid.

% motivation ISO dynamics
The dynamics \eqref{eq:ISOprimaldual} can be interpreted as
follows. If generator $i$ bids higher than the Lagrange multiplier
$\l$ (which can be interpreted as a price) associated with the power
balance constraint \eqref{eq:pow-bal-con}, then the power generation
setpoint at node $i$ is decreased, and vice versa. The terms $\rho
\1\1^T(P_g-P_d)$ and $-\sigma^2\w$ in~\eqref{eq:ISOprimaldual} help to
compensate for the supply-demand mismatch in the network.

In the following, we analyze the equilibria and the stability of the
interconnection of the physical power network
dynamics~\eqref{eq:swingeqU} with the bidding
process~\eqref{eq:bidding-process}. We assume that the bids and power
generations are initialized within the feasible domain, i.e.,
$b(0)\geq0, P_g(0)\geq0$.
%
%\margints{The projection of $b$ is never active if $b(0)\geq0$.}
%

\subsection{Equilibrium analysis of the interconnected
  system}\label{sec:equilibrium-analysis}

The closed-loop system composed of the ISO-generator bidding scheme \eqref{eq:bidding-process} and the power network dynamics \eqref{eq:swingeqU} is described by
\begin{subequations}\label{eq:cl-sys}
  \begin{align}
    \dot \vp&=D_t^T\w\label{eq:vp-dyn}
    \\
    M\dot \w&=-D_t\nabla U(\vp)-A\w+P_g-P_d\label{eq:w-dyn}
    \\
    \tau_b \dot b&=\pproj{P_g-\nabla C^*(b)}{b}\label{eq:b-dyn}
    \\
    \tau_g\dot P_g&=\pproj{\1\l-b+\rho \1\1^T (P_d-P_g) -\sigma^2\w}{P_g}\label{eq:Pg-dyn}
    \\
    \tau_\l\dot \l&=\1^T(P_d-P_g)\label{eq:l-dyn}
  \end{align}
\end{subequations}
where $C^*(b):=\sum_{i\in\until{n}}C_i^*(b_i)$, $\tau_{b}=\diag(\tau_{b1},\ldots,\tau_{bn})\in\mathbb R^{n\times n}$.
% equilibria
We investigate the equilibria of \eqref{eq:cl-sys}. In particular, we
are interested in equilibria that correspond simultaneously to an
efficient Nash equilibrium, economic dispatch and frequency
regulation, as specified next.
\begin{definition}[Efficient equilibrium]
  An equilibrium $\bar x=\col(\bar \vp,\bar \w,\bar b,\bar P_g,\bar
  \l)$ of \eqref{eq:cl-sys} is \emph{efficient} if $\bar \w=0$, $\bar
  b$ is an efficient Nash equilibrium, and $\bar P_g$ is a primal
  optimizer of \eqref{eq:ISO-OPF}.
\end{definition}
The next result shows that all equilibria of \eqref{eq:cl-sys} are
efficient.

\begin{proposition}\longthmtitle{Equilibria are
    efficient}\label{prop:eq-are-optimal-Nash}
  Any equilibrium $\bar x=\col(\bar \vp,\bar \w,\bar b,\bar P_g,\bar
  \l)$ of \eqref{eq:cl-sys} is efficient.
\end{proposition}
\begin{proof}
  Let $\bar x$ be an equilibrium of \eqref{eq:cl-sys}, then there
  exist $\bar \mu_b,\bar \mu_g\in\real^n$ such that
  \begin{subequations}\label{eq:cl-sys-eq}
    \begin{align}
      0&=D_t^T\bar \w\label{eq:vp-dyn-eq}
      \\
      0&=-D_t\nabla U(\bar \vp)-A\bar \w+\bar
      P_g-P_d\label{eq:w-dyn-eq}
      \\
      0&=\bar P_g-\nabla C^*(\bar b)+\bar\mu_b\label{eq:b-dyn-eq}
      \\
      0&=\1\bar\l-\bar b+\bar\mu_g\label{eq:Pg-dyn-eq}
      \\
      0&=\1^T(P_d-\bar P_g)\label{eq:l-dyn-eq}
      \\
      0&\leq \bar b\perp \bar\mu_b\geq0\label{eq:b-comp-slack-eq}\\
      0&\leq\bar P_g\perp \bar \mu_g\geq0\label{eq:Pg-comp-slack-eq}
    \end{align}
  \end{subequations}
  % \bar w=0
  We first show that $\bar \w=0$. From \eqref{eq:vp-dyn-eq} it follows
  that $\bar \w=\1\w_s$ for some $\w_s\in\real$. Then by
  pre-multiplying \eqref{eq:w-dyn-eq} by $\1^T$ and using
  \eqref{eq:l-dyn-eq} we obtain $\1^TA\1\w_s=0$, which implies that
  $\bar \w=\1\w_s=0$. We prove next that $\bar{P}_g$ is a primal
  optimizer of \eqref{eq:ISO-OPF}.
% optimality of P_g:
% Observe that the power balance $\1^T\bar P_g=\1^TP_d$ is obtained
% from \eqref{eq:l-dyn-eq}.
  We claim that $\bar \mu_b=0$ since, by contraction, if $\bar
  \mu_{bi}>0$ for some $i\in[n]$, then $\bar b_i=0$ and therefore
  $0=\bar P_{gi}-\nabla C_i^*(\bar b_i)+\bar \mu_{bi}=\bar P_{gi}+\bar
  \mu_{bi}>0$, see also Remark~\ref{rem:pos-eff-NE}. Therefore,
  \eqref{eq:b-dyn-eq} implies that $\bar P_g=\nabla C^*(\bar
  b)=\argmax_{P_{g}\geq0}\{P_g^T\bar b-C(P_g)\}$ and thus satisfies
  the optimality conditions
  \begin{align}\label{eq:KKTcon-eff-bid-C*}
    \begin{aligned}
      \nabla C(\bar P_g)&=\bar b+\bar \eta, \quad 0\leq \bar P_g\perp\bar\eta\geq 0,
    \end{aligned}
  \end{align}
  for some $\bar \eta$. Let us define $\bar
  \mu=\bar b+\bar \eta-\1\bar \l\geq 0$ where the inequality holds by
  \eqref{eq:Pg-dyn-eq}.  By \eqref{eq:Pg-comp-slack-eq} and
  \eqref{eq:KKTcon-eff-bid-C*} we have $\bar P_g^T\bar \mu=\bar
  P_g^T(\bar b-\1\bar \l)=P_g^T\bar \mu_g=0$.
  % Next, define $\bar \eta\geq 0$ as $\bar \eta=\bar b-\1
  % \l^*$ %which is satisfies $\bar \eta\geq 0$ by assumption.
  Hence, $(\bar P_g,\bar \l,\bar \mu)$ satisfies % the KKT conditions
  % of
  % \eqref{eq:ISO-OPF}
  % given by
  \begin{align}\label{eq:KKTcon-sc2}
    \begin{aligned}
      \n C(\bar P_g)&=\1\bar \l+\bar \mu, \qquad \1^T\bar P_g=\1^TP_d,
      \\
      0&\leq \bar P_g\perp \bar \mu\geq 0,
    \end{aligned}
  \end{align}
  implying that $(\bar P_g,\bar \l,\bar \mu)$ is a primal-dual
  optimizer of \eqref{eq:ISO-OPF}.  Furthermore,
  \eqref{eq:KKTcon-eff-bid-C*} implies $\bar b\leq\nabla C(\bar P_g)$
  and thus, by Proposition~\ref{prop:exis-NE}, $\bar b$ is an
  efficient Nash equilibrium. Hence, $\bar x$ is an efficient
  equilibrium of~\eqref{eq:cl-sys}.
\end{proof}

\subsection{Convergence analysis}
% In this section we establish local asymptotic stability of the
% interconnected system
% \eqref{eq:cl-sys}. %As it is formulated as a complementarity system, we use an appropriate invariance principle from \cite{brogliato2005krakovskii}, which we reformulate below.
% We refocus our attention to the closed-loop system \eqref{eq:cl-sys}
% or \eqref{eq:cl-sys-compact}.  We define the set of stable
% equilibria $\mathcal X\subset K$ as
% \begin{align*}
%   \mathcal X:=\{&\col(\bar \vp,\bar\w,\bar b,\bar P_g,\bar \l) : \eqref{eq:cl-sys-eq}\text{ holds for some }\bar\mu_b,\bar\mu_g\\&\text{and }D^TD_t^{\dagger T}\bar\vp\in(-\pi/2,\pi/2)^m\}
% \end{align*}
% the collection of $\bar x=\col(\bar \vp,\bar\w,\bar b,\bar P_g,\bar
% \l)$ that satisfy \eqref{eq:cl-sys-eq} and $D^TD_t^{\dagger
% T}\bar\vp\in(-\pi/2,\pi/2)^m$ for some $\bar \Lambda$.
In this section we establish the local asymptotic convergence
of~\eqref{eq:cl-sys} to an efficient equilibrium.

\begin{theorem}
  Consider the subset of (efficient) equilibria,
  \begin{align*}
    \mathcal X:=\{& \bar{x} = \col(\bar \vp,\bar\w,\bar b,\bar
    P_g,\bar \l) : \bar{x} \text{ is an equilibrium
      of~\eqref{eq:cl-sys}}
    % \eqref{eq:cl-sys-eq}\text{ holds for some
    % }\bar\mu_b,\bar\mu_g
    \\&\text{and }D^TD_t^{\dagger
      T}\bar\vp\in(-\pi/2,\pi/2)^m\}.
  \end{align*}
  Then $\mathcal X$ is locally asymptotically stable under
  \eqref{eq:cl-sys}. Moreover, the convergence is to a point.
\end{theorem}
\begin{proof}
  Our proof strategy to show local convergence to~$\mathcal X$ is
  based on applying Theorem~\ref{thm:inv-prin}, which is a special
  case of the invariance principle stated in
  \cite{brogliato2005krakovskii} adapted for complementarity
  systems. To this end, we rewrite the projected dynamical
  system~\eqref{eq:cl-sys} as the equivalent complementarity
  system~\eqref{eq:cl-comp-sys}, see also \cite[Theorem
  1]{brogliato2006equivalence} for more details,
  \begin{subequations}\label{eq:cl-comp-sys}
    \begin{align}
      \dot \vp&=D_t^T\w\label{eq:vp-comp-dyn}
      \\
      M\dot \w&=-D_t\nabla U(\vp)-A\w+P_g-P_d\label{eq:w-comp-dyn}
      \\
      \tau_b \dot b&=P_g-\nabla C^*(b)+\mu_b\label{eq:b-comp-dyn}
      \\
      \tau_g\dot P_g&=\1\l-b+\rho \1\1^T (P_d-P_g) -\sigma^2\w+\mu_g\label{eq:Pg-comp-dyn}
      \\
      \tau_\l\dot \l&=\1^T(P_d-P_g)\label{eq:l-comp-dyn}\\
      0&\leq  b\perp \mu_b\geq0\label{eq:b-comp-slack}\\
      0&\leq P_g\perp  \mu_g\geq0\label{eq:Pg-comp-slack}
    \end{align}
  \end{subequations}
  where $\mu_b,\mu_g\in\mathbb R^n$. We can equivalently write \eqref{eq:cl-comp-sys} in the compact form
  \begin{subequations}\label{eq:cl-sys-compact}
    \begin{align}
      \dot x&=F(x)+C^T\Lambda\\
      0&\leq Cx+d\perp \Lambda\geq0\label{eq:cl-sys-compact-slackness}
    \end{align}
  \end{subequations}
  with $x=\col(\vp,\w,b,P_g,\l), \Lambda=\col(\mu_b,\mu_g)$,  and 
  \begin{subequations}
    \begin{align}
      F(x)&=
      \begin{bmatrix}
        D_t^T\w \\
        M^{-1}(-D_t\nabla U(\vp)-A\w+P_g-P_d)\\
        \tau_b^{-1}( P_g-\nabla C^*(b))\\
        \tau_g^{-1}(\1\l-b+\rho\1\1^T (P_d-P_g) -\sigma^2\w)\\
        \tau_\l^{-1}\1^T(P_d-P_g)
      \end{bmatrix}\label{eq:F}\\
      C&=
      \begin{bmatrix}
        0           & 0          &     \tau_b^{-1} & 0         &0&0  \\
        0           & 0        &    0           & \tau_g^{-1} &0&0
      \end{bmatrix}, \quad d=0.     \label{eq:C-Lambda}
    \end{align}
  \end{subequations}
  Note that $F$ is Lipschitz continuous\footnote{Here we observe that, since $C$ is continuously differentiable and $\mu$-strongly convex on $\realnonnegative$, $C^*$ is $\frac{1}{\mu}$-Lipschitz continuous on $\realnonnegative$.}.  For the equivalence of the projected dynamical system \eqref{eq:cl-sys} and the complementarity system \eqref{eq:cl-comp-sys} to hold, we consider absolutely continuous solutions $t\mapsto x(t)$ that satisfy \eqref{eq:cl-comp-sys} almost everywhere (in time) in the sense of Lebesgue measure. In addition, we consider (unique) solutions of \eqref{eq:cl-sys-compact} that are \emph{slow}. That is, at each time $t$, $\Lambda$ satisfies \eqref{eq:cl-sys-compact-slackness} and is such that $\dot x(t)$ is of minimal norm, see also \cite{brogliato2006equivalence}.

  % --------------------------------------------------
  % FOR MY OWN UNDERSTANDING
  % First, observe that the system \eqref{eq:cl-sys} can be written
  % compactly in the form \eqref{eq:cl-sys} which identical as
  % \cite[eq. (3)]{brogliato2005krakovskii} or equivalently
  % \cite[eq. (1)]{brogliato2005krakovskii}, see also
  % \cite{brogliato2006equivalence} for more details on this
  % equivalence between complementarity systems and evolutionary
  % variational inequalities\footnote{In particular, notice that under
  % Slater's condition the associated projected dynamical system, the
  % evolutionary variational inequality and the complementarity system
  % are equivalent by Theorem 1 of \cite{brogliato2006equivalence}.}.
  % --------------------------------------------------

  Let $\bar x\in \mathcal X$ be arbitrary and fixed for the remainder
  of the proof. For aesthetic reasons we first consider the case where $\sigma=1$ in \eqref{eq:Pg-dyn} or \eqref{eq:Pg-comp-dyn} and later we explain how to generalize the convergence result. Consider the function $V$ defined by
  \begin{align}\label{eq:V}
    V(x)=U(\vp)-(\vp-\bar \vp)^T\n U(\bar \vp)-U(\bar
    \vp)+\frac12||x-\bar x||^2_{\tau}
  \end{align}
  with $\tau=\blockdiag(0,M,\tau_b,\tau_g,\tau_\l)$.  Note that
  $V(\bar x)=0,\n V(\bar x)=0$ and, since $D^TD_t^{\dagger
    T}\bar\vp\in(-\pi/2,\pi/2)^{m}$, $\n^2V(\bar x)>0$. Consequently,
  there exists a compact level set $\Psi$ of $V$ around $\bar x$. We
  show now that the two conditions of Theorem~\ref{thm:inv-prin} are
  satisfied.
  
  \textbf{Condition (I):} For $C$ given in \eqref{eq:C-Lambda} and
  $d=0$ the polyhedron \eqref{eq:K} takes the form
  \begin{align*}
    K=\{x=\col(\vp,\w,b,P_g,\l) : b\geq0, P_g\geq0\}.
  \end{align*}
  Consequently, for all $x\in\p K\cap \Psi$ we have
  % We claim that $x-\nabla V(x)\in K$ for all $x\in\p K\cap\Psi$
  % where $K$ is defined in \eqref{eq:K}. Observe that $b-\frac{\p
  % V}{\p b}(x)=b-\tau_b(b-\bar b)=\tau_b\bar b\geq0$ for all $x\in\p
  % K\cap \Psi$. Likewise, $P_g-\frac{\p V}{\p
  % P_g}(x)=P_g-\tau_g(P_g-\bar P_g)=\tau_g\bar P_g\geq0$ for all
  % $x\in\p K\cap \Psi$. Finally, since $\vp,\w,\l$ are unconstrained,
  % it follows that $x-\nabla V(x)\in K$ for all $x\in\p K\cap\Psi$.
  \begin{align*}
    x-\nabla V(x)=
    \begin{bmatrix}
      \vp-\nabla U(\vp)+\n U(\bar \vp)\\
      \w-M\w\\
      b-\tau_b(b-\bar b)\\
      P_g-\tau_g(P_g-\bar P_g)\\
      \l-\tau_\l(\l-\bar \l)
    \end{bmatrix}=
    \begin{bmatrix}
      *\\
      *\\
      \tau_b\bar b\\
      \tau_g\bar P_g\\
      *
    \end{bmatrix}\in K.
  \end{align*}

  \textbf{Condition (II):} Since $\bar x\in\mathcal X$ there exists
  $\bar \Lambda$ such that $F(\bar x)+C^T\bar \Lambda=0.$ As a result,
  for each $x\in K$ we have
  \begin{align}\label{eq:cond2}
    \begin{aligned}
      & \langle \nabla V(x),F(x) \rangle=\langle\nabla V(x),
      F(x)-F(\bar x)-C^T\bar \Lambda \rangle
      \\
      &=(\n U(\vp)-\n U(\bar \vp))D^T\w
      \\
      &+\w^T(-D(\n U(\vp)-\n U(\bar \vp))-A\w+P_g-\bar P_g)
      \\
      &+(b-\bar b)^T(P_g-\n C^*(b)-\bar P_g+\n C^*(\bar b)-\bar \mu_b)
      \\
      &+(P_g-\bar P_g)^T(\1(\l-\bar\l)-b+\bar b+\rho \1\1^T (\bar
      P_g-P_g)
      \\
      &-\sigma^2\w-\bar\mu_g)+(\l-\bar \l)\1^T(\bar P_g-P_g)
      \\
      &=-\w^TA\w-(b-\bar b)^T(\n C^*(b)-\n C^*(\bar b))
      \\
      &-\rho\|\1^T(\bar P_g-P_g)\|^2-(b-\bar b)^T\bar \mu_b
      \\
      &-(P_g-\bar P_g)^T\bar \mu_g\leq0
    \end{aligned}
  \end{align}
  where the inequality holds because $C^*$ is convex, $\bar b^T\bar
  \mu_b=0, \bar P_g^T\bar \mu_g=0$ and $\bar \mu_b,\bar
  \mu_g,b,P_g\geq0$. Hence, the second condition of Theorem~\ref{thm:inv-prin} is satisfied.

  \textbf{Invariance of $\Psi$:} We note that \eqref{eq:cond2} does
  not necessarily imply that $\Psi$ is forward invariant. We show this
  next. Observe that for each $x,\Lambda$ satisfying $0\leq Cx\perp
  \Lambda\geq0$ we have
  \begin{align}\label{eq:inv-Psi}
    \begin{aligned}
      & \langle \n V(x),F(x)+C^T\Lambda\rangle=\langle \n
      V(x),F(x)\rangle
      \\
      &+\langle \n V(x), C^T\Lambda\rangle\leq \langle \n V(x),
      C^T\Lambda\rangle
      \\
      &=(b-\bar b)^T \mu_b+(P_g-\bar P_g)^T \mu_g
      \\
      &=-\bar b^T \mu_b-\bar P_g^T \mu_g\leq 0.
    \end{aligned}
  \end{align}
  Hence, the $V$ is non-increasing along trajectories initialized in
  $K\cap\Psi$. Since $\Psi$ is a level set of $V$, this implies that
  $\Psi$ is forward invariant.
  
  \textbf{Largest invariant set:} Define
  \begin{align*}
    E=\{x\in K\cap \Psi \ : \ \langle F(x),\nabla V(x) \rangle=0\}
  \end{align*}
  and denote the largest invariant subset of $E$ by $\mathcal M$. By
  \eqref{eq:cond2} we note that each $x\in \mathcal M$ satisfies
  $\w=0, \1^T(P_d-P_g)=0$ and, $b_i=\bar b_i>0$ (otherwise, if $\bar
  b_i=0$, then $0=\bar P_{gi}-\n C_i^*(\bar b_i)+\bar\mu_{bi}=\bar
  P_{gi}+\bar\mu_{bi}>0$, which results in a contradiction) for each
  $i\in[n]$ with $\bar P_{gi}>0$ as $C_i^*$ is strictly convex around
  such $\bar b_i$. For these $i$, $P_{gi}=\bar P_{gi}>0$ by
  \eqref{eq:b-dyn} and $b_i=\l=\bar \l$ by \eqref{eq:Pg-dyn}.  For
  each $x\in\mathcal M$ and $i\in[n]$ with $\bar P_{gi}=0$, we have
  $\n C_i^*(b_i)=\n C_i^*(\bar b_i)=0$ by the convexity of $C_i$ and
  thus $P_{gi}=\bar P_{gi}=\mu_{b_i}=0$ and thus $b_i=\l+\mu_{gi}$.
  Hence, $\mathcal M\subset \mathcal X$ and therefore each trajectory
  initialized in $\Psi$ converges to $\mathcal X$. Moreover,
  from~\eqref{eq:inv-Psi}, we deduce that $\bar{x}$ is stable. Since
  this equilibrium has been chosen arbitrarily, we conclude that every
  point in $\mathcal X$ is Lyapunov stable, implying that convergence
  of the trajectories is to a point.

  The proof for the case $\sigma>0,\sigma\neq 1$ proceeds in the same
  way as before except that we appropriately scale the Lyapunov
  function. Specifically, we define the Lyapunov function $V$ as in
  \eqref{eq:V} but with
  $\tau=\blockdiag(0,M,\sigma\tau_b,\sigma\tau_g,\sigma\tau_\l)>0$.
\end{proof}

\section{Simulations}\label{sec:simulations}

\tikzstyle{edge} = [draw,line width=1pt,-]
\tikzstyle{edge} = [draw,line width=1pt,-]
\tikzstyle{vertex}=[circle,% ball color=contarea,
draw=black!100,line width=1pt,fill=blue!20,minimum size=16pt,inner sep=0pt]
\tikzstyle{block}=[rectangle,% ball color=contarea,
draw=black!100,line width=1pt,fill=blue!30,minimum width=18mm,rounded corners]
\begin{figure}[]
  \centering
  \begin{tikzpicture}[scale=1.0, auto,swap]% scale=4
    \definecolor{v1}{rgb}{0 ,   0.4470,    0.7410} 
    \definecolor{v2}{rgb}{    0.8500 ,   0.3250,    0.0980}
    \definecolor{v3}{rgb}{    0.9290 ,   0.6940,    0.1250}
    \definecolor{v4}{rgb}{    0.4940 ,   0.1840,    0.5560}
    \definecolor{v5}{rgb}{    0.4660 ,   0.6740,    0.1880}
    \definecolor{v6}{rgb}{    0.3010 ,   0.7450,    0.9330}
    \definecolor{v9}{rgb}{    0.6350 ,   0.0780,    0.1840}
    \definecolor{v8}{rgb}{    0.4940 ,   0.1840,    0.5560}
    \definecolor{v7}{rgb}{    1 ,   0.3,    0.3}
    % Draw the IEEE 14-bus benchmark network
    % First we draw the vertices
    \foreach \x/\y/\i/\c in {0/0/1/1,3/0/2/2,6/0/3/3,4.5/1/4/7,1.5/1/5/7,0/2/6/6,6/2/7/7,6/1/8/8,4.5/2
      /9/7,3/2/10/7,1.5/2/11/7,0/3/12/7,1.5/3/13/7,4.5/3/14/7},
    \node[vertex,fill=v\c!60] (\i) at (\x,\y) {$\i$};

    % Connect vertices with edges and draw weights
    \foreach \source/\dest in {1/2,1/5,2/3,2/4,2/5,3/4,4/5,4/7,4/9,5/6,6/11,6/12,6/13,7/8,7/9,9/10,9/14,10/11,12/13,13/14},
    \path[edge] (\source) -- (\dest);
    % \node[block] (load) at (6.5,2) {loads} ;
    % \node[block,fill=red!30] (gen) at (6.5,1) {generators} ;
  \end{tikzpicture}%\vspace*{-2mm}
  
  \caption{Schematic of the modified IEEE 14-bus benchmark. Each edge
    in the graph represents a transmission line.  Red nodes represent
    loads.  All the other nodes represent synchronous generators, with
    different colors that match the ones used in
    Figure~\ref{fig:proj-sys}. The physical dynamics are modeled
    by~\eqref{eq:swingdeltacomp}.  }
  \label{fig:14-bus}\vspace*{-1ex}
\end{figure}
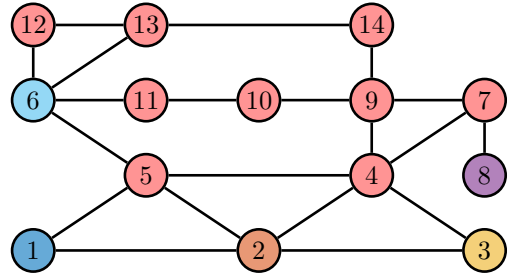

We simulate the closed-loop dynamics \eqref{eq:cl-sys} for the
modified IEEE 14-bus benchmark model illustrated in Figure
\ref{fig:14-bus}. We assume quadratic costs at each node $i\in[14]$ of
the form
\begin{align*}
  C_i(P_{gi})=\frac12q_iP_{gi}^2+c_iP_{gi}
\end{align*}
with $q_i>$ and $c_i\geq0$. In the original 14-bus model, nodes
$1,2,3,6,8$ have synchronous generators while the other nodes are
load nodes and have no power generation. We replicate this by
increasing the cost (by setting $q_i,c_i\gg0$) at the load nodes to
ensure positive power generation is not profitable at them.  In
addition, we choose $M_i\in[4,5.5]$ for generator nodes
$i\in\{1,2,3,6,8\}$ and $M_i\ll1$ for the load nodes. We set $
A_i\in[2,3], V_i\in[1,1.06]$ for all $i\in[n]$ and $\rho=300$. At
$t=\SI{0}{\second}$, the load (in \SI{}{\mega\watt}'s) is given by
\begin{align*}
  P_d= (0, 22, 80, 48, 7.6, 11, 0, 0, 30, 9.0, 3.5, 6.1, 14, 15).
\end{align*}
Initially, we set $(q_1,q_2,q_3,q_6,q_8)=(26,70,150,150,300)$ and
$(c_1,c_2,c_3,c_6,c_8)=(7.5,30,90,82.5,75)$. The
system~\eqref{eq:cl-sys} is initialized at steady state at the optimal
generation level
\begin{align*}
  (P_{g1},P_{g2})=(201, 44)
\end{align*}
and with $P_{gi}=0$ for all other nodes.  Figure~\ref{fig:proj-sys}
shows the evolution of the system in the case when $\sigma=300$ and
Figure \ref{fig:proj-sys0} in the case when $\sigma=0$. Note that in
the latter case, there is no frequency signal fed back into the
bidding process, so the dynamics~\eqref{eq:cl-sys} effectively becomes
a cascaded system (where the bidding process drives the physical
dynamics of the power network).  At $t=\SI{1}{\second}$ the load at
node $3$ is increased from \SI{80}{\mega\watt} to
\SI{94.2}{\mega\watt} and the trajectories converge to a new efficient
equilibrium with optimal power generation level
$(P_{g1},P_{g2})=(211,48)$ and $P_{gi}=0$ for all other
nodes. Furthermore, at steady state generators $1,2$ bid equal to the
Lagrange multiplier while generators $3,6,8$ bid their marginal cost
at zero production ($b_{i}=c_i$, for $i=3,6,8$) and thus, by
Proposition~\ref{prop:exis-NE}, we know that this corresponds to an
efficient Nash equilibrium.

\begin{figure*}[tbh]
  \centering
  \begin{subfigure}[t]{0.3\textwidth}
    \includegraphics[width=\textwidth]{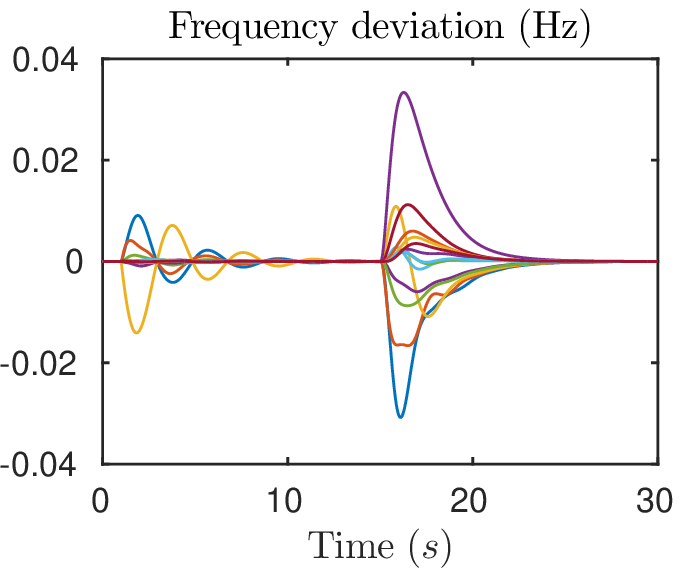}
    \caption{Evolution of the frequency deviation. After each
      change of the load or the cost functions, the frequency is
      restored to its nominal value.}
    \label{fig:freq}
  \end{subfigure}
  ~ %add desired spacing between images, e. g. ~, \quad, \qquad, \hfill etc. 
  % (or a blank line to force the subfigure onto a new line)
  \begin{subfigure}[t]{0.3\textwidth}
    \includegraphics[width=\textwidth]{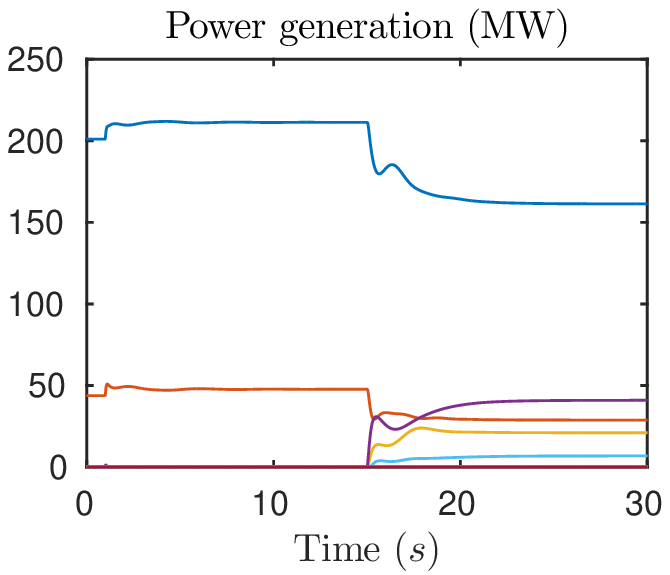}
    \caption{Evolution of the power generation at each node. After the change
      of the cost functions in nodes $3,6,8$, there is more
      competition among the generators, resulting in lower power
      productions at node $1$ and $2$.}
    \label{fig:Pg}
  \end{subfigure}
  ~ %add desired spacing between images, e. g. ~, \quad, \qquad, \hfill etc. 
  % (or a blank line to force the subfigure onto a new line)
  \begin{subfigure}[t]{0.3\textwidth}
    \includegraphics[width=\textwidth]{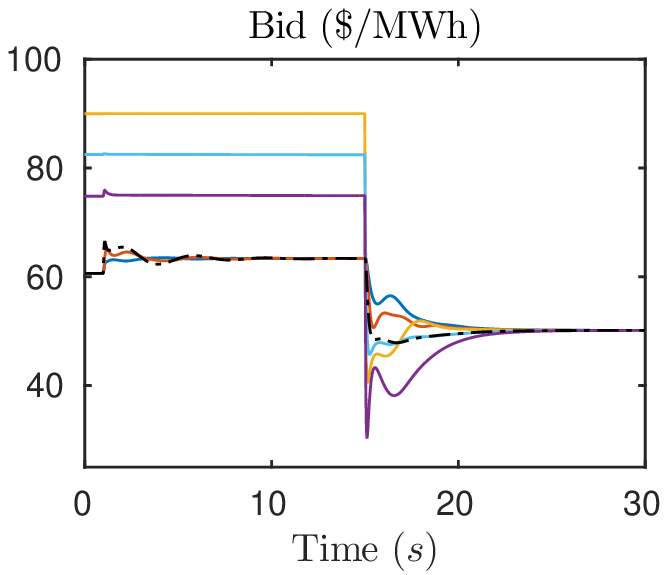}
    \caption{Evolution of the bids and the Lagrange multiplier
      (represented by the dashed black colored line). Initially,
      the marginal costs (and the bids) at zero power production
      are higher than the market equilibrium price for nodes
      $3,6,8$.}
    \label{fig:bids}
  \end{subfigure}
  \caption{Simulations of the interconnection~\eqref{eq:cl-sys}
    between the ISO-generation bidding mechanism and the power network
    dynamics with $\sigma=300$.  At $t=\SI{1}{\second}$ the load at
    node $3$ is increased from \SI{80}{\mega\watt} to
    \SI{94.2}{\mega\watt}. At $t=\SI{15}{\second}$ the marginal cost
    decreases at nodes $3,6,8$ which allows these generators to make
    profit by lowering their bids to have a positive power production
    as illustrated in plots (b) and (c). }\label{fig:proj-sys}
\end{figure*}

\begin{figure*}[tbh]
  \centering
  \begin{subfigure}[t]{0.3\textwidth}
    \includegraphics[width=\textwidth]{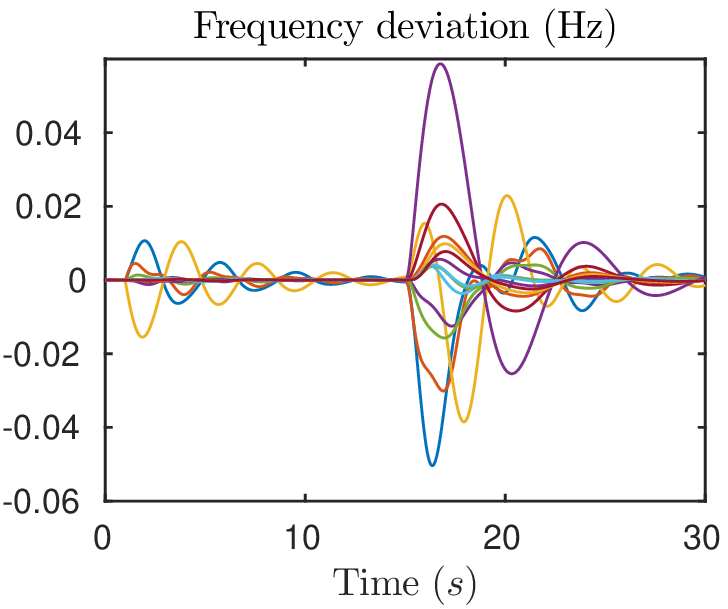}
    \caption{Evolution of the frequency deviation. Compared to
      Figure~\ref{fig:proj-sys}(a), there are more oscillations and a
      larger overshoot of the frequency deviations. }
    \label{fig:freq0}
  \end{subfigure}
  ~ %add desired spacing between images, e. g. ~, \quad, \qquad, \hfill etc. 
  % (or a blank line to force the subfigure onto a new line)
  \begin{subfigure}[t]{0.3\textwidth}
    \includegraphics[width=\textwidth]{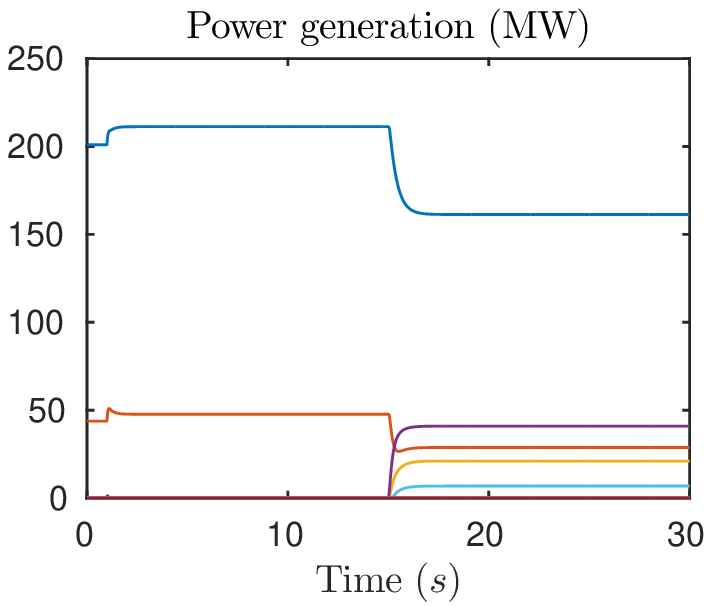}
    \caption{Evolution of the power generations at each node. As $\sigma=0$, the generations are updated only according to the bidding process without using any frequency measurements. }
    \label{fig:Pg0}
  \end{subfigure}
  ~ %add desired spacing between images, e. g. ~, \quad, \qquad, \hfill etc.
  % (or a blank line to force the subfigure onto a new line)
  \begin{subfigure}[t]{0.3\textwidth}
    \includegraphics[width=\textwidth]{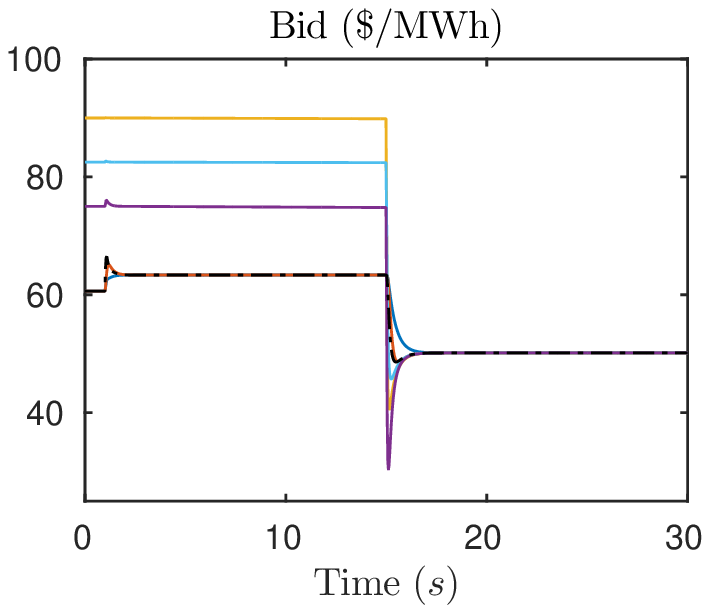}
    \caption{Evolution of the bids and the Lagrange multiplier.
      Compared to Figure~\ref{fig:proj-sys}(c), the convergence is
      faster in this scenario because the bidding process does not
      take into account its impact on the dynamics of the power
      network.}
    \label{fig:bids0}
  \end{subfigure}
  \caption{Simulations of the interconnection~\eqref{eq:cl-sys}
    between the ISO-generation bidding mechanism and the power network
    dynamics for the case of $\sigma=0$, i.e., when there is no
    frequency feedback signal in the bidding process. The scenario is
    the same as in Figure \ref{fig:proj-sys}. As illustrated, the
    closed-loop system also converges in this case to an efficient
    equilibrium.}\label{fig:proj-sys0}
\end{figure*}

At $t=\SI{15}{\second}$ the cost of producing electricity is decreased
in areas $3,6,8$ by setting $(q_3,q_6,q_8)=(60,75,68)$ and
$(c_3,c_6,c_8)=(38,45,23)$. This allows these generators to make
profit by participating in the bidding process and results in a
reduction of the total cost of the generation from
\SI{9711}{\$\per\hour} to \SI{8540}{\$\per\hour}. As illustrated in
both Figures~\ref{fig:proj-sys} and~\ref{fig:proj-sys0}, the power
generations converge to the new optimal steady state given by
\begin{align*}
  (P_{g1},P_{g2},P_{g3},P_{g6},P_{g8})=(161, 29, 21, 7, 41).
\end{align*}
In addition, we observe that after each change of either the load or
the cost function, the frequency is stabilized and the bids converge
to a new efficient Nash equilibrium. The fact that the frequency
transients are better in Figure~\ref{fig:proj-sys} than in
Figure~\ref{fig:proj-sys0} is consistent, since in the latter case
there is no frequency feedback in the bidding process.

\section{Conclusions}\label{sec:conclusions}
 
We have studied a market-based power dispatch scheme and its
interconnection with the swing dynamic of the physical network. From
the market perspective, we have considered a continuous-time bidding
scheme that describes the negotiation process between the independent
system operator and a group of competitive generators.
% We have studied the interaction between a competitive electricity
% market and the physical dynamics of the power network.  From the
% market perspective, we have considered a continuous-time bidding
% scheme that describes the negotiation process between the
% independent system operator and a group of strategic
% generators. % which respects the nonnegativity constraints
% On the physical side, we have considered the frequency dynamics
% described by the swing equations.
Using the frequency as a feedback signal in the bidding dynamics, we
have shown that the interconnected projected dynamical system provably
converges to an efficient Nash equilibrium (where generation levels
minimize the total cost) and to zero frequency deviation. In this way,
competitive generators are enabled to participate in the real-time
electricity market without compromising efficiency and stability of
the power system. Future work will investigate finite-horizon
scenarios and incorporate generator bounds and power flow constraints
in the economic dispatch formulation.

%%%%%%%%%%%%%%%%%%%% APPENDIX %%%%%%%%%%%%%%%%%%%%

% if have a single appendix:
% \appendix[Proof of the Zonklar Equations]
% or
% \appendix  % for no appendix heading
% do not use \section anymore after \appendix, only \section*
% is possibly needed

\appendix

% %
%  we use an appropriate invariance principle from \cite{brogliato2005krakovskii}, which we reformulate below

\begin{theorem}[Invariance principle for complementarity systems
  \cite{brogliato2005krakovskii}]
  \label{thm:inv-prin}
  Consider the system
\begin{subequations}\label{eq:compact-comp}
  \begin{align}
    \dot x&= F(x)+C^T\Lambda\\
    0&\leq Cx+d\perp \Lambda\geq0
  \end{align}
\end{subequations}
  with  Lipschitz continuous $F$ and let $K$ be the polyhedron 
\begin{align}\label{eq:K}
  K=\{x : Cx+d\geq 0\}.
\end{align}
Let $\Psi\subset\real^n$ be a compact set and $\map{V}{\real^n}{\real}$ be a continuous differentiable function such that 
\begin{align*}
\begin{tabular}[h]{cll}
\textup{(\bf I)}&\quad $x-\n V(x)\in K$, &for all $x\in\p K\cap \Psi$,\\
\textup{(\bf II)}&\quad $\langle \n V(x),F(x)\rangle\leq0$, &for all $x\in K\cap \Psi$.
\end{tabular}
\end{align*}
Let $E\subset \mathbb R^n$ be given by
\begin{align*}
  E:=\{x\in K\cap \Psi : \langle F(x),\n V(x)\rangle=0\}
\end{align*}
and denote the largest invariant subset of $E$ by $\mathcal M$. Then, for each $x_0\in K$ such that its orbit satisfies $\gamma(x_0)\subset \Psi$, we have 
\begin{align*}
  \lim_{t\to\infty}d(x(t;t_0,x_0),\mathcal M)=0.
\end{align*}
\end{theorem}
%

%\appendices
%\renewcommand{\theequation}{A.\arabic{equation}}
%\renewcommand{\thetheorem}{A.\arabic{theorem}}
%\setcounter{section}{0}
%\renewcommand{\thesection}{Appendix \Alph{section}}

%%%%%%%%%%%%%%%%%%%%%%%%%%%%%%%%%%%%%%%%%%%%%%%%%%

%%%%%%%%%%%%%%%%%%%%%%%%%%%%%%%%%%%%%%%%%%%%%%%%%% 

% \section*{Acknowledgment}

\bibliographystyle{IEEEtran}
\bibliography{IEEEabrv,bibfile}

% Generated by IEEEtran.bst, version: 1.14 (2015/08/26)
\begin{thebibliography}{10}
\providecommand{\url}[1]{#1}
\csname url@samestyle\endcsname
\providecommand{\newblock}{\relax}
\providecommand{\bibinfo}[2]{#2}
\providecommand{\BIBentrySTDinterwordspacing}{\spaceskip=0pt\relax}
\providecommand{\BIBentryALTinterwordstretchfactor}{4}
\providecommand{\BIBentryALTinterwordspacing}{\spaceskip=\fontdimen2\font plus
\BIBentryALTinterwordstretchfactor\fontdimen3\font minus
  \fontdimen4\font\relax}
\providecommand{\BIBforeignlanguage}[2]{{%
\expandafter\ifx\csname l@#1\endcsname\relax
\typeout{** WARNING: IEEEtran.bst: No hyphenation pattern has been}%
\typeout{** loaded for the language `#1'. Using the pattern for}%
\typeout{** the default language instead.}%
\else
\language=\csname l@#1\endcsname
\fi
#2}}
\providecommand{\BIBdecl}{\relax}
\BIBdecl

\bibitem{PSCC18-TS-AC-CDP-AS-JC}
T.~Stegink, A.~Cherukuri, C.~D. Persis, A.~van~der Schaft, and J.~Cort\'es,
  ``Stable interconnection of continuous-time price-bidding mechanisms with
  power network dynamics,'' in \emph{Power Systems Computation Conference},
  Dublin, Ireland, 2018, submitted.

\bibitem{trip2016internal}
S.~Trip, M.~B{\"u}rger, and C.~{D}e Persis, ``An internal model approach to
  (optimal) frequency regulation in power grids with time-varying voltages,''
  \emph{Automatica}, vol.~64, pp. 240--253, 2016.

\bibitem{zhangpapaautomatica}
X.~Zhang and A.~Papachristodoulou, ``A real-time control framework for smart
  power networks: Design methodology and stability,'' \emph{Automatica},
  vol.~58, pp. 43--50, 2015.

\bibitem{li2016connecting}
N.~Li, C.~Zhao, and L.~Chen, ``Connecting automatic generation control and
  economic dispatch from an optimization view,'' \emph{IEEE Transactions on
  Control of Network Systems}, vol.~3, no.~3, pp. 254--264, 2016.

\bibitem{cady2015distributed}
S.~T. Cady, A.~D. Dom{\'\i}nguez-Garc{\'\i}a, and C.~N. Hadjicostis, ``A
  distributed generation control architecture for islanded {AC} microgrids,''
  \emph{IEEE Transactions on Control Systems Technology}, vol.~23, no.~5, pp.
  1717--1735, 2015.

\bibitem{dorfler2016breaking}
F.~D{\"o}rfler, J.~W. Simpson-Porco, and F.~Bullo, ``Breaking the hierarchy:
  Distributed control and economic optimality in microgrids,'' \emph{IEEE
  Transactions on Control of Network Systems}, vol.~3, no.~3, pp. 241--253,
  2016.

\bibitem{alv_meng_power_coupl_market}
F.~Alvarado, J.~Meng, C.~DeMarco, and W.~Mota, ``Stability analysis of
  interconnected power systems coupled with market dynamics,'' \emph{IEEE
  Transactions on Power Systems}, vol.~16, no.~4, pp. 695--701, 2001.

\bibitem{DJS-MC-AMA:16}
D.~J. Shiltz, M.~Cvetkovi\'{c}, and A.~M. Annaswamy, ``An integrated dynamic
  market mechanism for real-time markets and frequency regulation,'' \emph{IEEE
  Transactions on Sustainable Energy}, vol.~7, no.~2, pp. 875--885, 2016.

\bibitem{stegink2017unifying}
T.~Stegink, C.~De~Persis, and A.~van~der Schaft, ``A unifying energy-based
  approach to stability of power grids with market dynamics,'' \emph{IEEE
  Transactions on Automatic Control}, vol.~62, no.~6, pp. 2612--2622, 2017.

\bibitem{cherukuri2017iterative}
A.~Cherukuri and J.~Cort{\'e}s, ``Iterative bidding in electricity markets:
  rationality and robustness,'' \emph{arXiv preprint arXiv:1702.06505}, 2017,
  submitted to \emph{IEEE Transactions on Control of Network Systems}.

\bibitem{cherukuri2016decentralized}
------, ``Decentralized {Nash} equilibrium learning by strategic generators for
  economic dispatch,'' in \emph{American Control Conference (ACC), 2016}.\hskip
  1em plus 0.5em minus 0.4em\relax IEEE, 2016, pp. 1082--1087.

\bibitem{powsysdynwiley}
J.~Machowski, J.~Bialek, and J.~Bumby, \emph{Power System Dynamics: Stability
  and Control}, 2nd~ed.\hskip 1em plus 0.5em minus 0.4em\relax Ltd: John Wiley
  \& Sons, 2008.

\bibitem{TB-GJO:82}
T.~Ba{\c s}ar and G.~Oldser, \emph{Dynamic Noncooperative Game Theory}.\hskip
  1em plus 0.5em minus 0.4em\relax Academic Press, 1982.

\bibitem{DF-JT:91}
D.~Fudenberg and J.~Tirole, \emph{Game Theory}.\hskip 1em plus 0.5em minus
  0.4em\relax Cambridge, MA: MIT Press, 1991.

\bibitem{hiriart2013convex}
J.-B. Hiriart-Urruty and C.~Lemar{\'e}chal, \emph{Convex analysis and
  minimization algorithms I: Fundamentals}.\hskip 1em plus 0.5em minus
  0.4em\relax Springer science \& business media, 2013, vol. 305.

\bibitem{cherukuri2017saddle}
A.~Cherukuri, B.~Gharesifard, and J.~Cortes, ``Saddle-point dynamics:
  conditions for asymptotic stability of saddle points,'' \emph{SIAM Journal on
  Control and Optimization}, vol.~55, no.~1, pp. 486--511, 2017.

\bibitem{goebel2017stability}
R.~Goebel, ``Stability and robustness for saddle-point dynamics through
  monotone mappings,'' \emph{Systems \& Control Letters}, vol. 108, pp. 16--22,
  2017.

\bibitem{brogliato2005krakovskii}
B.~Brogliato and D.~Goeleven, ``The {Krakovskii-LaSalle} invariance principle
  for a class of unilateral dynamical systems,'' \emph{Mathematics of Control,
  Signals and Systems}, vol.~17, no.~1, pp. 57--76, 2005.

\bibitem{brogliato2006equivalence}
B.~Brogliato, A.~Daniilidis, C.~Lemar{\'e}chal, and V.~Acary, ``On the
  equivalence between complementarity systems, projected systems and
  differential inclusions,'' \emph{Systems \& Control Letters}, vol.~55, no.~1,
  pp. 45--51, 2006.

\end{thebibliography}

% biography section
% 
% If you have an EPS/PDF photo (graphicx package needed) extra braces are
% needed around the contents of the optional argument to biography to prevent
% the LaTeX parser from getting confused when it sees the complicated
% \includegraphics command within an optional argument. (You could create
% your own custom macro containing the \includegraphics command to make things
% simpler here.)
% \begin{IEEEbiography}[{\includegraphics[width=1in,height=1.25in,clip,keepaspectratio]{mshell}}]{Michael Shell}
%   or if you just want to reserve a space for a photo:

% \begin{IEEEbiography}{Tjerk Stegink}
%   Biography text here.
% \end{IEEEbiography}

% \begin{IEEEbiography}{Ashish Cherukuri}
%   Biography text here.
% \end{IEEEbiography}

% \begin{IEEEbiography}{Claudio De Persis}
%   Biography text here.
% \end{IEEEbiography}

% \begin{IEEEbiography}{Arjan van der Schaft}
%   Biography text here.
% \end{IEEEbiography}

% \begin{IEEEbiography}{Jorge Cort\'{e}s}
%   Biography text here.
% \end{IEEEbiography}

\end{document}